\documentclass [12pt,reqno]{amsart}
\usepackage{times}
\usepackage{mathtools}
\usepackage[toc,page]{appendix} 
\usepackage{bm}
\usepackage{pifont}
\numberwithin{equation}{subsection}

\allowdisplaybreaks
\usepackage{amssymb,latexsym,amsmath,amsfonts, eucal}
\usepackage[usenames,dvipsnames,svgnames,table]{xcolor}

\newtheorem{thm}{Theorem}[section]
      \newtheorem{lemma}[thm]{Lemma}

      \newtheorem{example}[thm]{Example}
      
      \newtheorem{rmk}[thm]{Remark}
      
      \numberwithin{equation}{section}
       \newtheorem{prob}[thm]{Problem}
\title [Composition operators]{ Composition operators on de Branges spaces of entire functions}

\author[Garg]{Bharti Garg$^1$}

\address{$^{1,3}$
	Department of Mathematics\\
	Indian Institute of Technology Ropar\\
	140001\\
	India}

\author[Mahapatra]{Subhankar Mahapatra$^2$}

\address{$^2$
	School of Mathematical Sciences\\
	National Institute of Science Education and Research Bhubaneswar\\
	752050\\
	India}

\author[Sarkar]{Santanu Sarkar$^3$}

\begin{document}

\keywords{composition operators, model spaces in the upper half-plane, Hardy Hilbert spaces, de Branges spaces, regular de Branges spaces\\
$1$ Email: bharti.20maz0012@iitrpr.ac.in, bhartigargfdk@gmail.com\\
$2$ Email: subhankar@niser.ac.in, subhankarmahapatra95@gmail.com\\
$3$ Corresponding author; Email: santanu@iitrpr.ac.in, santanu87@gmail.com.}

\begin{abstract}
This paper aims to study the boundedness and compactness of composition operators from model spaces to the Hardy Hilbert spaces in the upper half-plane. Consequently, we investigate the boundedness and compactness of composition operators on de Branges spaces of entire functions. Moreover, we observe that the boundedness of a composition operator on a regular de Branges space forces the inducing symbol to be affine; conversely, affine symbols under appropriate conditions yield bounded composition operators. Furthermore, we show that the behaviour of boundedness and compactness of composition operators on general de Branges spaces is different from that on the Paley-Wiener spaces.
\end{abstract}
\maketitle

\tableofcontents
\section{Introduction}\label{sec1} 

The theory of composition operators on the Hardy Hilbert spaces and model spaces on the unit disc $\mathbb{D}$ has been extensively studied; see, for instance, \cite{Cowen, Shapiro, Manhas, Nordgren} and the references therein. Mashreghi and Shabankhah in \cite{MS1, MS2} initiated the study of composition operators on model spaces in $\mathbb{D}$. Lyubarskii and Malinnikova in \cite{LM} studied compact composition operators from model spaces to the Hardy Hilbert space $H^2(\mathbb{D})$. By Littlewood's subordination principle, every analytic self map of $\mathbb{D}$ induces a bounded composition operator on the Hardy spaces. However, the situation changes when one moves from the unit disc to the half-plane setting. Matache in \cite{Matache} showed that the composition operator is bounded on $H^2(\mathbb{C}_+)$ if and only if it has a finite angular derivative at infinity. Subsequent work by Elliott and Jury in \cite{SamJury} further developed this theory, providing a detailed characterization of boundedness and compactness of composition operators on Hardy spaces in the right half-plane setting. However, the study of composition operators from model spaces to the Hardy Hilbert space $H^2(\mathbb{C}_+)$ in the upper half-plane has received less attention.

Composition operators on spaces of entire functions are also actively studied. Chacon and Gimenez in \cite{Chacon} studied composition operators on Paley-Wiener spaces and showed that the composition operator $C_{\phi}$ is bounded on such spaces if and only if $\phi$ is an affine map. Paley-Wiener spaces are a particular example of de Branges spaces. This naturally raises the question of whether a similar characterization holds in a more general setting of de Branges spaces. These spaces of entire functions were studied by de Branges in \cite{Brange} and are closely connected to the spectral theory, differential equations, and prediction theory; see, for instance \cite{DymAdv, DMcK}. Among them, regular de Branges spaces have additional structural stability and play an important role in prediction theory. Recently, Bellavita in \cite{Bellavita1, Bellavita2} investigated the boundedness of translation operators on de Branges spaces. His work focuses on a specific class of composition operators, in particular, vertical and horizontal translation operators. This leaves the question of understanding the general composition operators on such spaces.

The study of composition operators on de Branges spaces is natural from both
the operator-theoretic and function-theoretic points of view. Since de Branges
spaces are Hilbert spaces of entire functions whose structure is determined by
the Hermite-Biehler function \(E\), the boundedness of a composition operator
\(C_\phi\) on \(H(E)\) is expected to depend not only on the growth of the
entire symbol \(\phi\), but also on the interaction between \(\phi\) and the
function \(E\). Thus, composition operators provide a way to understand how
changes of variable preserve, or fail to preserve, the geometry of a de Branges
space. This is particularly relevant because de Branges spaces arise naturally
in spectral theory, differential equations and prediction theory. Therefore,
studying boundedness and compactness of composition operators on these spaces
helps clarify how the analytic and spectral structure encoded by \(E\) behaves
under composition.

 At this point, it is useful to compare the expected behaviour in the
Paley-Wiener setting with the situation in general de Branges spaces.
For Paley-Wiener spaces, Chacon and Gimenez proved that a composition
operator induced by a non-constant entire symbol is bounded if and only if
the symbol is affine of the form \(\phi(z)=az+b\), with \(a\in\mathbb R\)
and \(0<|a|\leq 1\). Thus, in that setting, boundedness is completely
determined by the particular affine form of the symbol. In the present paper, we show that this affine
rigidity persists for regular de Branges spaces: if \(H(E)\) is regular
and \(C_\phi\) is bounded for a non-constant entire symbol \(\phi\), then
\(\phi\) must be affine. However, the general de Branges setting is not
identical to the Paley-Wiener setting. Indeed, the boundedness of
\(C_\phi\) on \(H(E)\) also depends on the interaction between the symbol
\(\phi\) and the Hermite-Biehler function \(E\). In particular, we show
that an affine symbol satisfying the Paley-Wiener condition need not
induce a bounded composition operator on an arbitrary de Branges space.
Moreover, while bounded composition operators on Paley-Wiener spaces are
never compact, compact composition operators may exist on suitable
non-Paley-Wiener de Branges spaces. 

The present paper is devoted to the study of boundedness and compactness of composition operators $C_{\phi}$ from model spaces to $H^2(\mathbb{C}_+)$. Consequently, we investigate composition operators on de Branges spaces
of entire functions, with particular emphasis on entire symbols satisfying
natural boundedness conditions. Our results extend some aspects of the earlier works of Chacon and Gimenez \cite{Chacon} on Paley-Wiener spaces and of Bellavita \cite{Bellavita1, Bellavita2} on translation operators, and provide an initial framework for studying composition operators on de Branges spaces beyond the Paley-Wiener
setting.
 
The plan of the paper is as follows. A brief introduction and preliminaries required for this article are covered in Sections \ref{sec1} and \ref{sec pre}, respectively. In Section \ref{sec2}, we study composition operators from
model spaces in the upper half-plane into \(H^2(\mathbb C_+)\), obtaining
boundedness and non-compactness results. In Section \ref{sec3}, we transfer these
ideas to de Branges spaces and provide examples which distinguish general
de Branges spaces from Paley-Wiener spaces. In Section \ref{sec4}, using techniques
inspired by Chacon and Gimenez \cite{Chacon}, we prove an affine rigidity
result for regular de Branges spaces and give sufficient conditions under
which affine symbols induce bounded composition operators.\\
The following notations will be used throughout this paper: 
\begin{itemize} 
 \item $\Im(z)$ and $\Re(z)$ denote the imaginary and real parts of the complex number $z$, respectively.
 \item $\rho_\xi(z)= -2 \pi i (z-\bar{\xi})$.
 \item $\ominus$ denotes the orthogonal complement.
\item $\mathbb{N}$, $\mathbb{R}$, $\mathbb{C},$ $\mathbb{C}_+$, and $\mathbb{C}_-$ denote the set of natural numbers, the real line, the complex plane, the open upper half-plane, and the open lower half- plane, respectively.

\item $\Vert \cdot \Vert_e$ denotes the essential norm of a bounded operator in a Hilbert space, that is, for any bounded operator $A$ in a Hilbert space $\mathcal{H}$,
$$\Vert A \Vert_e=\inf\{||A-Q||:~Q~\mbox{is compact in}~\mathcal{H}\}.$$

\item $f^{\#}(z):= \overline{f(\bar{z})}$.
\item $R_z$ denotes the generalized backward shift operator defined by \begin{equation}
    (R_z g)(\xi) := \left\{
   \begin{array}{ll}
         \frac{g(\xi)-g(z)}{\xi-z}  & \mbox{if }~ \xi \neq z \vspace{0.1cm} \\ 
        g'(z) & \mbox{if }~ \xi = z
   \end{array} \right.
\end{equation}
for every $z,~ \xi \in \mathbb{C}$.
\end{itemize}

\section{Preliminaries} \label{sec pre}
In this section, we recall basic definitions and results that we shall use in this paper. Let $E$ be an entire function from the Hermite-Biehler class $\mathcal{H}B$, that is, $E$ satisfies the following inequality: $$\vert E(\bar{z})\vert < \vert E(z) \vert~ \text{ for all}~z \in \mathbb{C}_+.$$
Then the de Branges space of entire functions corresponding to $E \in \mathcal{H}B$ is defined as follows:
$$H(E):= \{f~\text{entire}: \frac{f}{E}, \frac{f^{\#}}{E}\in H^2(\mathbb{C}_+) \},$$
where $H^2(\mathbb{C}_+)$ is the Hardy Hilbert space on the upper half-plane. The de Branges space $H(E)$ is endowed with the following inner product:
$$\langle f,g \rangle_{H(E)}= \langle \frac{f}{E}, \frac{g}{E} \rangle_{H^2(\mathbb{C}_+)}= \int_{-\infty}^{\infty}  f(t) \overline{g(t)}\frac{1}{\vert E(t) \vert^2}  dt.$$
The space $H(E)$ is a reproducing kernel Hilbert space corresponding to the reproducing kernel $$K_w(z):= \left\{
    \begin{array}{ll}
         \frac{ E(z)\overline{E(w)}- \overline{E(\bar{z})}E(\bar{w})}{-2 \pi i(z-\bar{w})}  & \text{if } z \neq \bar{w} \\ 
         \frac{E^{'} (\bar{w})\overline{E(w)}- E(\bar{w})\overline{E(w)}^{'}}{-2\pi i} & \text{if } z = \bar{w}.
    \end{array} \right.$$
For more details on de Branges spaces of entire functions, we refer to \cite{Brange}. Now, we recall some classical spaces of analytic functions. $H^{\infty}(\mathbb{C}_+)$ denotes the set of all bounded analytic functions on $\mathbb{C}_+$. $N(\mathbb{C}_+)$ denotes the set of analytic functions on $\mathbb{C}_+$ which can be represented as the quotient of two bounded analytic functions and referred to as functions of bounded type. If $f,g \in N(\mathbb{C}_+)$, then $f+g, fg \in N(\mathbb{C}_+)$. Moreover, if $g \not\equiv 0 $, and $\frac{f}{g}$ is holomorphic in $\mathbb{C}_+$ then $\frac{f}{g} \in N(\mathbb{C}_+).$ For $f \in N(\mathbb{C}_+)$, the mean type of $f$ is given as follows:
$$\mathrm{mt}(f):= \limsup_{y \rightarrow \infty} \frac{1}{y} \log \vert f(iy) \vert .$$
For any $f,g \in N(\mathbb{C}_+)$, $\mathrm{mt}(fg)= \mathrm{mt}(f)+\mathrm{mt}(g)$.
$N^+(\mathbb{C}_+)$ denotes the space of functions of bounded type such that in the inner-outer factorization of $f$, there is no singular function in the denominator. If $f,g \in N^+(\mathbb{C}_+)$, then $f+g, fg \in N^+(\mathbb{C}_+)$. If $f \in N^+(\mathbb{C}_+)$, then $\mathrm{mt}(f)$ is non-positive. Conversely, if $f \in N(\mathbb{C}_+)$ such that $f$ is of non-positive mean type and $f$ has continuous extension to the real axis, then $f \in N^+(\mathbb{C}_+)$. Indeed, every function \(f\in N(\mathbb C_{+})\) admits a Nevanlinna
factorization of the form
\[
    f(z)=e^{-iaz} O(z)B(z)\frac{S_{1}(z)}{S_{2}(z)},
\]
where \(a\) is the mean type of \(f\), \(O\) is an outer function, \(B\) is a Blaschke factor, and \(S_{1}\) and \(S_{2}\) are singular inner functions (see \cite[Theorem~6.13]{RR}). The continuous extension of \(f\) to the real axis eliminates the finite singular measure in the denominator of the Nevanlinna factorization, while the non-positive mean type eliminates the denominator exponential factor. This follows from the Stieltjes inversion formula (see \cite[Theorem~5.4]{RR}) and the Poisson representation (see \cite[Theorem~5.2]{RR}). This result is also discussed by Baranov et al. and Abakumov et al. in \cite{Baranov} and \cite{Abakumov}, respectively. 

The spaces satisfy the following inclusions:
$$H^{\infty}(\mathbb{C}_+), H^2(\mathbb{C}_+) \subset N^+(\mathbb{C}_+) \subset N(\mathbb{C}_+).$$ For more details on these spaces and proof of these results, see \cite[Chapter 5-6]{RR}.

Equivalently, an entire function $f \in H(E)$ if and only if \begin{itemize}
\item[(1)] $\frac{f}{E}$ and $\frac{f^{\#}}{E}$ are of bounded type
and non-positive mean type,
\item[(2)] $\int_{-\infty}^{\infty} \big\vert  \frac{f(t)}{E(t)}\big \vert^2 dt < \infty.$
\end{itemize}
Yet another equivalent axiomatic definition of de Branges spaces is defined as follows \cite[Theorem 23]{Brange}:
A reproducing kernel Hilbert space $H$ of entire functions is called a de Branges space if it satisfies the following two conditions:
\begin{itemize}
\item[(1)] If $f \in H$, then $f^{\#} \in H$ and $\langle f^{\#}, g^{\#}\rangle= \langle f,g \rangle$ for all $f,g \in H$.
\item[(2)] If $w \in \mathbb{C}\setminus \mathbb{R}$ and $ f \in H$ such that $f(w)=0$, then $\frac{z- \bar{w}}{z-w}f(z) \in H$ and $\langle \frac{z- \bar{w}}{z-w}f(z), \frac{z- \bar{w}}{z-w}g(z)\rangle = \langle f,g \rangle$ for all $f, g \in H$ such that $f(w)=0=g(w)$.
\end{itemize}

The de Branges space $H(E)$ is said to be regular (or short) if $H(E)$ is closed under the map $R_{\alpha}$ for every complex number $\alpha$. For more details, see \cite[Section 6.2]{DMcK}. Recall that an analytic function $\chi$ in $\mathbb{C}_+$ is said to be an inner function if $\vert \chi(z) \vert \leq 1$ for all $z \in \mathbb{C}_+$ and $\vert \chi(x) \vert =1$ a.e. on $\mathbb{R}$. Corresponding to this inner function $\chi$, the model space $H(\chi)$ is defined as follows: $H(\chi):= H^2(\mathbb{C}_+) \ominus \chi H^2(\mathbb{C}_+).$
The following theorem provides the correspondence between de Branges spaces of entire functions and model spaces.
\begin{thm} \cite[Theorem 2.1]{Baranov}\label{thmD}
Let $E$ be an entire function in the class $\mathcal{H}B$. Then the map $f \mapsto \frac{f}{E} $ is a unitary operator from $H(E)$ onto $H(\chi)$, where $\chi= \frac{E^{\#}}{E}$.
\end{thm}

Next, we recall the definition of meromorphic inner functions. An inner function $\Theta$ is said to be meromorphic if $\Theta$ coincides in $\mathbb{C}_+$ with a meromorphic function whose poles are in $\mathbb{C}_-$. Any meromorphic function $\Theta$ can be represented by $\Theta(z)= \frac{E^{\#}(z)}{E(z)}$ with a suitable entire function $E(z)$ in the class $\mathcal{H}B$ and having zeros only in the lower half-plane. Moreover, the following theorem provides another characterization of such functions.
\begin{thm}\cite[Lemma 2.1]{Baranov}\label{thmC}
Let $\chi$ be an inner function in $\mathbb{C}_+$. Then the following are equivalent:
\begin{itemize}
\item[(1)] $\chi=\frac{E^{\#}}{E}$, where $E$ is an entire function in the class $\mathcal{H}B$.
\item[(2)] $\chi(z)=B(z) \exp(i \alpha z)$, $z \in \mathbb{C}_+, $ where $\alpha\geq 0$ and $B$ is a Blaschke product such that the sequence of its zeros have no limit point in $\mathbb{C}$.
\end{itemize}
\end{thm} For more details on meromorphic inner functions, we refer to \cite{Javad}. Now, we recall the definitions of order and type of an entire function. Let $f$ be an entire function. Define $M_f(r):= \max_{\vert z \vert=r} \vert f(z) \vert$. Then, $f$ is said to be of order $\alpha$ if $$\limsup_{r \rightarrow \infty} \frac{\log \log M_f(r)}{\log r}= \alpha.$$ If the entire function $f$ is of positive finite order $\alpha$, then it is said to be of type $\beta$ if $$\limsup_{r \rightarrow \infty} \frac{\log M_f(r)}{r ^{\alpha}}= \beta.$$ An entire function $f$ is said to be of exponential type $\beta$ if it is of order one and has finite positive type $\beta$; and then we write $\mathrm{ET}(f)= \beta$. For more details on the order and type of entire functions, we refer to \cite{Boas, Levin}. Next, we recall the Polya theorem that we shall use in Section \ref{sec4}.
 
\begin{thm}\cite{Polya}\label{polya} Let $f$ and $g$ be entire functions such that $f \circ \phi$ is of finite order. Then exactly one of the following two conditions hold:
\begin{itemize}
\item[(1)] $f$ is of finite order and $g$ is a polynomial.
\item[(2)] $f$ is of order $0$ and $g$ is not a polynomial.  
\end{itemize}
\end{thm}
Let us recall the following theorem by M. G. Krein that we shall use later.
\begin{thm}\label{thmA} \cite[Theorem 6.17]{RR}
Let $f(z)$ be an entire function. Then the following are equivalent:
\begin{itemize}
\item [(1)] $f$ is of exponential type and  $$\int_{-\infty}^{\infty} \frac{\log^+ \vert f(t)\vert}{1+t^2}dt < \infty.$$
\item[(2)] The restrictions of $f(z)$ and $f^{\#}(z)$ to the upper half-plane belong to $N(\mathbb{C}_+).$
\end{itemize}
\end{thm}
The next theorem provides the relationship between the mean type and the exponential type of an entire function.
\begin{thm}\label{thmB}\cite[Theorem 6.18]{RR} 
Let $f \not\equiv 0$ be entire function satisfying equivalent conditions $(1)$ and $(2)$ of Theorem \ref{thmA}. Let $\mathrm{ET}(f)$ denote the exponential type of $f$. Let $\mathrm{mt_+}$ and $\mathrm{mt_-}$ denote the mean types of restriction of $f$ and $f^{\#}$ to the upper half-plane, respectively. Then, $\mathrm{mt_+}+ \mathrm{mt_-} \geq 0$ and $\mathrm{ET(f)}= \max \{ \mathrm{mt_+}, \mathrm{mt_-}\}.$ 
\end{thm}

Next, we recall the angular derivative at $\infty$ of an analytic function in the upper half-plane. A sequence of points $z_n:= x_n+iy_n$ in $\mathbb{C}_+$ is said to approach $\infty$ non-tangentially if $y_n \rightarrow \infty$ and the ratios $\frac{\vert x_n \vert}{y_n}$ are uniformly bounded. We say a map $\phi: \mathbb{C}_+ \rightarrow \mathbb{C}_+$ fixes $\infty$ non-tangentially if $\phi(z_n)\rightarrow \infty$ whenever $z_n \rightarrow \infty$ non-tangentially, and we write $\phi(\infty)= \infty$. Moreover, if the non-tangential limit $\lim_{z \rightarrow \infty} \frac{z}{\phi(z)}$ exists and is finite, then we say that $\phi$ has a finite angular derivative and write $\phi'(\infty)=\lim_{z \rightarrow \infty} \frac{z}{\phi(z)}$. For more details, see \cite{Boo}, \cite{Javadbook}, and \cite{SamJury} in the upper half-plane setting, the disc setting, and the right half-plane setting, respectively. The following is the Julia-Caratheodory theorem for the upper half-plane setting.

\begin{thm}\cite[Proposition 2.2]{Boo}\label{JC}
Let $\phi:\mathbb{C}_+ \rightarrow \mathbb{C}_+$ be holomorphic. Then the following are equivalent:
\begin{itemize}
\item[(1)] $\phi(\infty)= \infty$ and $\phi'(\infty)< \infty$.
\item[(2)] $\sup_{z \in \mathbb{C}_+} \frac{\Im(z)}{\Im(\phi(z))} < \infty$.
\item[(3)] $\limsup_{z \rightarrow \infty} \frac{\Im(z)}{\Im(\phi(z))} < \infty$.
\end{itemize}
In this case, quantities in $(2)$ and $(3)$ are both equal to $\phi'(\infty)$. 
\end{thm}

\section{Composition operators from model spaces to the Hardy Hilbert spaces in the upper half-plane}\label{sec2}
 
This section discusses the boundedness and compactness of composition operators from model spaces to Hardy Hilbert spaces in the upper half-plane. For any inner function $\chi$ in the upper half-plane $\mathbb{C}_+$, the model space $H(\chi)$ is defined as $H(\chi):= H^2(\mathbb{C}_+) \ominus \chi H^2(\mathbb{C}_+).$
\begin{thm}\label{lemma2.1}
Let $\chi$ be an inner function in $\mathbb{C}_+$ and $\phi : \mathbb{C}_+ \rightarrow \mathbb{C}_+$ be an analytic function. Then, for the composition operator $C_{\phi}: H(\chi) \rightarrow H^2(\mathbb{C}_+)$, the following statements hold:
\begin{itemize} 
\item[(1)] If $C_{\phi}$ is bounded, then
\begin{equation} \label{ineq2.1}
\sup_{z \in \mathbb{C}_+} \frac{\Im(z)}{\Im(\phi(z))}[1- \vert \chi(\phi(z)) \vert^2] < \infty.
\end{equation}
\item[(2)] If
\begin{equation} \label{eq3.2}
\sup_{z \in \mathbb{C}_+} \frac{\Im(z)}{\Im(\phi(z))} < \infty,
\end{equation}
then $C_{\phi}$ is bounded.
\end{itemize} 
Moreover, if the inner function $\chi$ and the analytic function $\phi$ are such that $$sup_{z \in \mathbb{C}_+} \vert \chi(\phi(z))\vert <1,$$ then the composition operator $C_{\phi}$ is bounded if and only if \eqref{eq3.2} holds.
\end{thm}

\begin{proof}
(1)~Let $K^{\chi}_w(z)$ and $K^{H^2}_w(z)$ represent the reproducing kernels of $H(\chi)$ and $H^2(\mathbb{C}_+)$ respectively. Let $C_{\phi}: H(\chi) \rightarrow H^2(\mathbb{C}_+)$ be bounded, i.e., there exists a constant $M$ such that $\Vert C_{\phi}\Vert = M$. The adjoint operator $C_{\phi}^*$ is bounded on $\operatorname{span}\{K^{H^2}_w, w \in \mathbb{C}_+\}$, thus $\Vert C_{\phi}^*K^{H^2}_z \Vert \leq M \Vert K^{H^2}_z \Vert.$ Since, $C_{\phi}^*K^{H^2}_z= K_{\phi(z)}^{\chi}$, we have the following inequalities:
\begin{eqnarray*}
&& \frac{1}{2 \pi} \frac{1-\vert \chi(\phi(z))\vert ^2}{2 \Im(\phi(z))} \leq M \frac{1}{2 \pi} \frac{1}{2 \Im(z)}  \\ 
&\Rightarrow &\frac{\Im(z)}{\Im(\phi(z))}(1-\vert \chi(\phi(z))\vert ^2) \leq M^2 \\
&\Rightarrow& \sup_{z \in \mathbb{C}_+} \frac{\Im(z)}{\Im(\phi(z))}(1-\vert \chi(\phi(z))\vert ^2)< \infty.
\end{eqnarray*}

(2)~For sufficiency, form a densely defined operator $C_{\phi}^*: K^{H^2}_w \rightarrow K^{\chi}_{\phi(w)}$. If $C_{\phi}^*$ is bounded on $\operatorname{span}\{K^{H^2}_w, w \in \mathbb{C}_+\}$, then it can be uniquely extended to a bounded operator on $H^2(\mathbb{C}_+)$. As, \begin{equation}
\langle (C_{\phi}^*)^*f, K^{H^2}_w \rangle_{H^2(\mathbb{C}_+)}= \langle f, C_{\phi}^* K^{H^2}_w \rangle= \langle f, K_{\phi(w)}^{H(\chi)}\rangle= f(\phi(w))= \langle C_{\phi}f, K^{H^2}_w\rangle
\end{equation}
for all $f \in H(\chi)$, so $C_{\phi}^*$ is the adjoint of $C_{\phi}$. So, in order to show that $C_{\phi}$ is bounded, it is sufficient to show that $C_{\phi}^*$ is bounded on $\operatorname{span}\{K^{H^2}_w, w \in \mathbb{C}_+\}$. Let $f= \sum_{i=1}^n c_i K^{H^2}_{w_i}$ and $\lambda = \sup\frac{\Im(z)}{\Im(\phi(z))}$. Define a function $L$ by $$L(z,w):= \lambda\langle K^{H^2}_w,K^{H^2}_z \rangle_{H^2(\mathbb{C}_+)}- \langle C_{\phi}^* K^{H^2}_w, C_{\phi}^* K^{H^2}_z \rangle_{H(\chi)}.$$ If $L$ is positive, i.e,\begin{eqnarray*}
&& \sum_{i,j=1}^{n} c_i \bar{c_j}L(w_i, w_j) \geq 0\\
&\Rightarrow& \sum c_i \bar{c_j} \lambda \langle K^{H^2}_{w_i}, K^{H^2}_{w_j} \rangle- \sum c_i \bar{c_j}  \langle C_{\phi}^* K^{H^2}_{w_i}, C_{\phi}^* K^{H^2}_{w_j} \rangle \geq 0\\
& \Rightarrow &  \lambda \Vert f \Vert^2 - \Vert C_{\phi}^*f\Vert^2 \geq 0.
\end{eqnarray*}
Hence, $C_{\phi}^*$ is bounded on $\operatorname{span}\{K^{H^2}_w, w \in \mathbb{C}_+\}$. Now it is left to show that $L$ is positive on $\mathbb{C}_+ \times
 \mathbb{C}_+$. 
 \begin{eqnarray*}
 L(z,w)&=& \lambda K^{H^2}(z,w)- K^{H(\chi)}(\phi(z), \phi(w))\\
&=& \lambda \frac{1}{-2 \pi i (z-\bar{w})}- \frac{1-\overline{\chi(\phi(w))}\chi(\phi(z))}{-2 \pi i (\phi(z)-\overline{\phi(w)})} \\
&=&\lambda \frac{1}{-2 \pi i (z-\bar{w})}- \frac{1}{-2 \pi i (\phi(z)-\overline{\phi(w)})} +\frac{\overline{\chi(\phi(w))}\chi(\phi(z))}{-2 \pi i (\phi(z)-\overline{\phi(w)})}
 \end{eqnarray*} 
 Let \begin{eqnarray*}
&& L_1(z,w) = \lambda \frac{1}{-2 \pi i (z-\bar{w})}- \frac{1}{-2 \pi i (\phi(z)-\overline{\phi(w)})}\\
 &\Rightarrow& \lambda^{-1}L_1(z,w)= \frac{1}{-2 \pi i (\phi(z)-\overline{\phi(w)})} \bigg[ \frac{\phi(z)- \lambda^{-1}z- (\overline{\phi(w)- \lambda^{-1}w})}{z-\bar{w}} \bigg]
 \end{eqnarray*}
 As, \begin{eqnarray*}
 &&\sup_{z \in \mathbb{C}_+} \frac{\Im(z)}{\Im(\phi(z))}=\lambda\\
 &\Rightarrow& \Im(z)\leq \lambda \Im(\phi(z))\\
 & \Rightarrow & \Im(\phi(z)- \lambda^{-1}z) \geq 0 
 \end{eqnarray*}  
 for all $z \in \mathbb{C}_+$. Thus, $L_1(z,w)$  is a positive scalar multiple of the pointwise product of two positive kernels, i.e., it is the pointwise product of the pull-back of the reproducing kernel of $H^2(\mathbb{C}_+)$ under $\phi$ given by $$\frac{1}{-2 \pi i (\phi(z)-\overline{\phi(w)})},$$ and the Nevanlinna kernel given by $$\frac{\phi(z)- \lambda^{-1}z- (\overline{\phi(w)- \lambda^{-1}w})}{z-\bar{w}}.$$ By
the Schur product theorem, the pointwise product of two positive kernels is again a positive kernel. Therefore $L_1$ is a positive kernel. Since the sum of two positive kernel functions is positive, we get that $L(z,w)$ is a positive kernel function.

Now, let $\chi$ and $\phi$ be such that $\sup_{z \in \mathbb{C}_+} \vert \chi(\phi(z))\vert <1$. This implies that $\vert \chi(\phi(z))\vert <1$ for all $z \in \mathbb{C}_+$. Hence, if $C_{\phi}$ is bounded, then from the inequality (\ref{ineq2.1}), we get that \begin{eqnarray*}
\frac{\Im(z)}{\Im(\phi(z))} &\leq& \frac{M^2}{1- \vert \chi(\phi(z)) \vert^2}\\
\Rightarrow \sup_{z \in \mathbb{C}_+} \frac{\Im(z)}{\Im(\phi(z))} &\leq& \sup_{z \in \mathbb{C}_+}\frac{M^2}{1- \vert \chi(\phi(z)) \vert^2}\\
            &=& \frac{M^2}{1- \sup_{z \in \mathbb{C}_+}\vert \chi(\phi(z)) \vert^2 } < \infty.
\end{eqnarray*}
\end{proof}

\begin{rmk}\label{remark}
By Theorem \ref{JC}, the sufficient condition stated in the above theorem is equivalent to the analytic self map $\phi$ having a finite angular derivative at infinity. 
\end{rmk}

Next, we shall discuss the compactness of the composition operators from model spaces to the Hardy Hilbert spaces in the upper half-plane.

\begin{thm}\label{thm2.7}
Let $\chi$ be a non constant inner function on $\mathbb{C}_+$ and $\phi : \mathbb{C}_+ \rightarrow \mathbb{C}_+$ be an analytic function such that $\sup_{z \in \mathbb{C}_+} \vert \chi(\phi(z))\vert <1$. Then any bounded composition operator $C_{\phi}: H(\chi) \rightarrow H^2(\mathbb{C}_+)$ is not compact.
\end{thm}

\begin{proof}
For given $\epsilon > 0$, there exists a compact operator $K$ such that $\Vert C_{\phi}^* \Vert_e+\epsilon \geq \Vert C_{\phi}^*-K \Vert$, where $\Vert C_{\phi}^* \Vert_e$ is the essential norm given by $$\Vert C_{\phi}^* \Vert_e= \inf \{ \Vert C_{\phi}^*- Q \Vert: Q~\text{is compact} \}.$$ 
Now,
\begin{eqnarray*}
\Vert C_{\phi}^*-K \Vert &\geq& \limsup_{z \rightarrow\infty} \frac{\Vert (C_{\phi}^*-K)K_z^{H^2} \Vert}{\Vert K_z^{H^2} \Vert}\\
&=& \limsup_{z \rightarrow\infty} \frac{\Vert C_{\phi}^*K_z^{H^2} \Vert}{\Vert K_z^{H^2} \Vert}\\
&=& \bigg(\limsup_{z \rightarrow\infty} \frac{\Im(z)}{\Im(\phi(z))}(1- \vert \chi(\phi(z))\vert^2)\bigg)^{1/2}\\
&\geq& \bigg(\limsup_{z \rightarrow\infty} \frac{\Im(z)}{\Im(\phi(z))}~\liminf_{z \rightarrow\infty} (1- \vert \chi(\phi(z))\vert^2)\bigg)^{1/2} >0.
\end{eqnarray*}
The second last equality follows from the compactness of $K$ and the fact that the normalized sequence $\frac{K_z^{H^2}}{\Vert K_z^{H^2} \Vert} \rightarrow 0$ weakly as $z \rightarrow \infty$ non-tangentially. Observe that since \(\phi:\mathbb C_+\to \mathbb C_+\) is analytic, we have $\Im z>0$ and $\Im(\phi(z))>0$ for every \(z\in \mathbb C_+\). Hence $
\frac{\Im z}{\Im(\phi(z))}>0$ for every \(z\in \mathbb C_+\), which implies that $\lambda:=\sup_{z\in\mathbb C_+}
\frac{\Im z}{\Im(\phi(z))}>0$. By the hypothesis of Theorem 3.3, and by the last part of Theorem 3.1, we have
$
\lambda<\infty.
$
Moreover, by the Julia-Caratheodory theorem in the upper half-plane (Theorem 2.6),
$
\lambda =\sup_{z\in\mathbb C_+}
\frac{\Im z}{\Im(\phi(z))}
=
\limsup_{z\to\infty}
\frac{\Im z}{\Im(\phi(z))}
=
\phi'(\infty).
$
Hence,
$
\limsup_{z\to\infty}
\frac{\Im z}{\Im(\phi(z))}>0.
$ 

This implies that the operator $C_{\phi}^*$, and hence the operator $C_{\phi}$, is not compact.
\end{proof}

In the following example, we examine the boundedness and compactness of composition operators from a model space associated with a subclass of meromorphic inner functions to the Hardy Hilbert space.

\begin{example}
Let $\chi(z)$ be a meromorphic inner function that is not a Blaschke product. Then there exists a Blaschke product $B(z)$, whose sequence of zeros has no limit point in $\mathbb{C}$, and a positive constant $\alpha(>0)$ such that \begin{equation}\label{chi}\chi(z)= B(z)\exp(i \alpha z).
\end{equation}
Let $\phi: \mathbb{C}_+ \rightarrow \mathbb{C}_+$ be an analytic function satisfying $\inf_{z \in \mathbb{C}_+} \Im(\phi(z))=d >0$. Then $\vert \chi(\phi(z)) \vert = \vert B(\phi(z)) \vert~ \vert \exp(i \alpha \phi(z))\vert \leq \vert \exp(i \alpha \phi(z))\vert$. Consequently,
\begin{eqnarray*}
\sup_{z \in \mathbb{C}_+}\vert \chi(\phi(z)) \vert 
&\leq& \sup_{z \in \mathbb{C}_+}\vert \exp(i \alpha \phi(z)) \vert\\
&=& \sup_{z \in \mathbb{C}_+} \exp(- \alpha \Im(\phi(z)))\\
&=& \exp(- \alpha \inf_{z \in \mathbb{C}_+} \Im(\phi(z)))
=\exp(-\alpha d)<1.
\end{eqnarray*}
 Hence, by Theorem \ref{lemma2.1}, the composition operator $C_{\phi}: H(\chi)\rightarrow H^2(\mathbb{C}_+)$ is bounded if and only if $$\sup_{z \in \mathbb{C}_+} \frac{\Im(z)}{\Im(\phi(z))} < \infty,$$
and, by Theorem \ref{thm2.7}, any such bounded operator is not compact.  

In particular, consider a vertical translation operator $\phi$ in $\mathbb{C}_+$ defined by \begin{equation} \label{phi}\phi(z)= z+ib, ~b>0. 
 \end{equation}
Then, $\phi(\mathbb{C}_+) \subseteq \mathbb{C}_+$ is an analytic function and $\inf_{z \in \mathbb{C}_+} \Im(\phi(z))=b >0$. Moreover, $$\sup_{z \in \mathbb{C}_+} \frac{\Im(z)}{\Im(\phi(z))} = \sup_{z \in \mathbb{C}_+} \frac{y}{y+b}=1.$$ Therefore, for $\chi(z)$ and $\phi(z)$ as given by \eqref{chi} and \eqref{phi}, respectively, the composition operator $C_{\phi}: H(\chi)\rightarrow H^2(\mathbb{C}_+)$ is bounded but not compact.
\end{example}

\section{Composition operators on de Branges spaces of entire functions}\label{sec3}

In this section, we discuss the boundedness and compactness of composition operators on de Branges spaces of entire functions. Let $\phi$ be an entire function such that $\phi(\mathbb{C}_+) \subseteq \mathbb{C}_+$ and $E$ be an entire function in the class $\mathcal{H}B$. By the closed graph theorem, the composition operator $C_{\phi}$ is bounded on $H(E)$ if and only if $f \circ \phi \in H(E)$, for all $f \in H(E)$, or equivalently, if $f \circ \phi$ is entire function, 
\begin{equation}  \label{defi}
\frac{f \circ \phi}{E} \in H^2(\mathbb{C}_+) \quad \mbox{and} \quad \frac{(f \circ \phi)^{\#}}{E} \in H^2(\mathbb{C}_+),~\mbox{for all}~f \in H(E).
\end{equation}
As, $\phi(\mathbb{C}_+) \subseteq \mathbb{C}_+$, we have that $\Im(\phi) \geq 0$ for all $z \in \mathbb{C}_+$. Consequently, $\phi$ admits a Nevanlinna representation on $\mathbb{C}_+$ of the form $$\phi(z)= b+cz+ \frac{1}{\pi} \int_{-\infty}^{\infty} \bigg( \frac{1}{t-z}- \frac{t}{1+t^2} \bigg) d\mu(t),$$ where $b \in \mathbb{R}, c \geq 0,$ and $\mu$ is non-negative Borel measure on $( -\infty, \infty )$ satisfying $$ \int_{-\infty}^{\infty} \frac{d \mu(t)}{1+t^2} < \infty.$$ Now, suppose that, $\phi^{\#}(z):=\overline{\phi(\bar{z})}= \phi(z)$ on $\mathbb{C}_+$. Also, observe that $(f\circ \phi)^{\#}= f^{\#} \circ \phi^{\#}$. Then, since both $E$ and $E \circ \phi$ have no zeros in $\mathbb{C}_+$, we may write $$\frac{f \circ \phi}{E}= \frac{E \circ \phi}{E}~ \frac{f \circ \phi}{E \circ \phi}$$ and $$\frac{(f \circ \phi)^{\#}}{E}=\frac{f^{\#} \circ \phi^{\#}}{E}= \frac{f^{\#} \circ \phi}{E}= \frac{E \circ \phi}{E}~\frac{f^{\#} \circ \phi}{E \circ \phi}.$$
Thus, if the following two conditions are satisfied
\begin{itemize}
\item[(1)] $\frac{E \circ \phi}{E} \in H^{\infty}(\mathbb{C}_+)$,
\item[(2)] $\frac{f \circ \phi}{E \circ \phi} \in H^2(\mathbb{C}_+), \frac{f^{\#} \circ \phi}{E \circ \phi} \in H^2(\mathbb{C}_+)$ for all $f \in H(E)$,
\end{itemize} then the operator $C_{\phi}$ is bounded on $H(E)$.
The following theorem provides an improved sufficient condition for the boundedness of $C_{\phi}.$

\begin{thm}\label{thm3.3}
Let $\phi$ and $E$ be entire functions such that $\phi(\mathbb{C}_+) \subseteq \mathbb{C}_+$, $\phi^{\#}= \phi$ on $\mathbb{C}_+$ and $E \in \mathcal{H}B$. If the following two conditions are satisfied
\begin{itemize}
\item[(1)] $\frac{E \circ \phi}{E} \in H^{\infty}(\mathbb{C}_+)$,
\item[(2)] $\sup_{z \in \mathbb{C}_+}\frac{\Im(z)}{\Im(\phi(z))} < \infty$, 
\end{itemize} then the operator $C_{\phi}$ is bounded on $H(E)$.
\end{thm}
\begin{proof}
Let $S= \{ \frac{f}{E} : f \in H(E) \} $ and $S^{\#}= \{ \frac{f^{\#}}{E} : f \in H(E) \} $. By definition of de Branges space $H(E)$, we have $S= S^{\#}$. Moreover, by Theorem \ref{thmD}, $S=H(\chi)$, where $\chi= \frac{E^{\#}}{E}$. The condition $(2)$ of the above discussion holds if and only if the operator $C_{\phi}:S \rightarrow H^2(\mathbb{C}_+)$ is bounded.  Now, the proof follows by using Theorem \ref{lemma2.1}.
\end{proof}

Now, we provide the necessary condition for the boundedness of the composition operator $C_{\phi}$ on the de Branges space $H(E)$.

\begin{thm}\label{thm3.4}
Let $\phi$ and $E$ be entire functions such that $E \in \mathcal{H}B$. If the operator $C_{\phi}$ is bounded on $H(E)$, then the following condition holds true:
$$\sup_{z \in \Lambda} \frac{\Im(z)}{\Im(\phi(z))} \bigg\vert \frac{E \circ \phi(z)}{E(z)} \bigg\vert^2 \bigg( \frac{1- \vert \chi(\phi(z)) \vert ^2}{1- \vert \chi(z) \vert ^2} \bigg)< \infty$$
where, $\chi= \frac{E^{\#}}{E}$ and $\Lambda= \{ z \in \mathbb{C}_+ : E(\phi(z)) \neq 0 \} $. Moreover, if the entire function $\phi$ is such that $\phi(\mathbb{C}_+) \subseteq \mathbb{C}_+$, then $\Lambda= \mathbb{C}_+$. 
\end{thm}
\begin{proof}
The proof follows similarly to the necessary part of  Theorem \ref{lemma2.1}.
\end{proof}


Next, we shall discuss the compactness of composition operators on de Brange spaces of entire functions $H(E)$. Let $t_n$ be the zeros of the real function $A(z)$ defined as follows:
\begin{equation} \label{A}
A(z)= \frac{E(z)+E^{\#}(z)}{2}.
\end{equation}
We recall the following theorem due to de Branges \cite{Brange} (it is also stated in \cite{Bellavitathesis}), which provides an orthonormal basis of the space $H(E)$. This result will be used to derive a condition under which a bounded composition operator on $H(E)$ fails to be compact.

\begin{thm}(see \cite[Theorem 22]{Brange}, \cite[Theorem 2.8]{Bellavitathesis}) If $A(z) \in H(E)$, then the set $\{ \frac{A(z)}{\Vert A(z) \Vert}\} \cup \{ \frac{k_{t_n}(z)}{\Vert k_{t_n}(z) \Vert} \}$, where $k_{t_n}$ are the reproducing kernels of the space $H(E)$ at the points $t_n$, forms an orthonormal basis of the space $H(E)$. If $A(z) \notin H(E)$, then the set $\{ \frac{k_{t_n}(z)}{\Vert k_{t_n}(z) \Vert} \}$ forms an orthonormal basis of $H(E)$.
\end{thm}

Observe that $\frac{k_{t_n}}{\Vert k_{t_n} \Vert}$ converges to $0$ weakly as $n$ tends to infinity. Indeed, for any $f \in H(E)$, $$ \langle f, \frac{k_{t_n}}{\Vert k_{t_n} \Vert} \rangle= \frac{f(t_n)}{\Vert k_{t_n} \Vert}$$
tends to $0$ by the Parseval's identity as $n \rightarrow \infty$. The following theorem provides a sufficient condition for a bounded composition operator to be non compact.

\begin{thm}\label{Thm4.4} Let $\phi$ and $E$ be entire functions such that $E \in \mathcal{H}B$. Let $C_{\phi}$ be a bounded composition operator on $H(E)$. If the following condition is satisfied:
$$d \leq \limsup_{n \rightarrow \infty} \left\{
    \begin{array}{ll}
          \frac{ \vert E(\phi(t_n))\vert ^2 - \vert E(\overline{\phi(t_n)})\vert^2}{-4 i \Im(\phi(t_n))E(t_n) \Re(E'(t_n))} & \text{if } \phi(t_n) \notin \mathbb{R} \\ 
             \frac{-i \Im(E'(\phi(t_n)) \overline{E(\phi(t_n))})}{E(t_n) \Re(E'(t_n))}& \text{if } \phi(t_n) \in \mathbb{R} 
    \end{array}, \right.$$ where $d$ is some finite positive constant and $t_n$ are the zeros of $A(z)$(as defined in \eqref{A}). Then the operator $C_{\phi}$ is not compact on $H(E)$.
\end{thm}
\begin{proof}
For given $\epsilon > 0$, there exists a compact operator $K$ such that $\Vert C_{\phi}^* \Vert_e+\epsilon \geq \Vert C_{\phi}^*-K \Vert$. 
Now,
\begin{eqnarray*}
\Vert C_{\phi}^*-K \Vert &\geq& \limsup_{n \rightarrow\infty} \Vert (C_{\phi}^*-K)\frac{k_{t_n} }{\Vert k_{t_n} \Vert}\Vert\\
&=& \limsup_{n \rightarrow\infty} \frac{\Vert C_{\phi}^*k_{t_n} \Vert}{\Vert k_{t_n} \Vert}\\
&=& \limsup_{n \rightarrow\infty} \frac{\Vert k_{\phi(t_n)} \Vert}{\Vert k_{t_n} \Vert} \\
&=& \limsup_{n \rightarrow \infty} \left\{
    \begin{array}{ll}
           \sqrt{\frac{ \vert E(\phi(t_n))\vert ^2 - \vert E(\overline{\phi(t_n)})\vert^2}{-4 i \Im(\phi(t_n))E(t_n) \Re(E'(t_n))}}  & \text{if } \phi(t_n) \notin \mathbb{R} \\ 
             \sqrt{\frac{-i \Im(E'(\phi(t_n)) \overline{E(\phi(t_n))})}{E(t_n) \Re(E'(t_n))}}& \text{if } \phi(t_n) \in \mathbb{R} 
    \end{array} \right.\\
    &\geq& \sqrt{d}>0.
\end{eqnarray*}
This implies that the operator $C_{\phi}^*$, and hence the operator $C_{\phi}$, is not compact.
\end{proof}

We now provide some examples illustrating the boundedness and compactness of composition operators on de Branges spaces.

\begin{example}\label{example3.6}
Let $\phi(z)=z+b$, where $b= b_1+ib_2 \in \mathbb{C}$ such that $b_2 \geq 0$. Clearly, $\phi$ is an entire function such that $\phi(\mathbb{C}_+) \subseteq \mathbb{C}_+$. Also, $\sup_{z \in \mathbb{C}_+} \frac{\Im(z)}{\Im(\phi(z))}= 1$. Now, let $E(z)= E(0) \exp(-idz)$ such that $d=d_1+id_2$ and $d_1 > 0$. It is easy to check that $E \in \mathcal{H}B$. Such an $E$ also belongs to the Polya class with no zeros \cite[Chapter 1, Section 7]{Brange}. Observe that 
$$
\bigg\vert \frac{E \circ \phi}{E}(z) \bigg\vert = \bigg\vert\frac{E(z+b)}{E(z)} \bigg\vert= \bigg\vert \frac{E(0)\exp(-id(z+b))}{E(0)\exp(-idz)} \bigg\vert=\vert \exp(-idb)\vert.
$$
Hence $\frac{E \circ \phi}{E} \in H^{\infty}(\mathbb{C}_+)$. Thus, in the case $b_2=0$, we have $\phi^{\#}=\phi$ on $\mathbb{C}_+$; consequently, Theorem \ref{thm3.3} implies that the composition operator $C_{\phi}$ is bounded on $H(E)$.

Now, we consider the case when $b_2>0$. By the argument used in the proof of Theorem \ref{thm3.3}, it suffices to show that $\frac{(f\circ \phi)^{\#}}{E}\in H^2(\mathbb{C}_+)$
for every $f\in H(E)$. This will establish the boundedness of $C_{\phi}$ on $H(E)$. For \(f\in H(E)\), we have \(\phi^\#(z)=z+\overline b\), and hence
\[
    \frac{(f\circ\phi)^\#(z)}{E(z)}
    =\frac{f^\#(z+\overline b)}{E(z)}
    =C e^{id_1z}\left(e^{-d_2(z+\overline b)}f^\#(z+\overline b)\right),
\]
where \(C\neq 0\) is a constant. Since \(e^{-d_2z}f^\#(z)\) is of exponential
type at most \(d_1\) and belongs to \(L^2(\mathbb R)\), the Plancherel-Polya
estimate (see \cite[Theorem 6.7.1]{Boas}, \cite[Lecture 7]{Levin}) gives
\[
\int_{\mathbb R}\left|\frac{(f\circ\phi)^\#(x+iy)}{E(x+iy)}\right|^2dx
\leq C e^{-2d_1y}e^{2d_1|y-b_2|}
\int_{\mathbb R}|e^{-d_2t}f^\#(t)|^2dt .
\]
Since \(e^{-2d_1y}e^{2d_1|y-b_2|}\) is uniformly bounded for \(y>0\), it follows that $\frac{(f\circ\phi)^\#}{E}\in H^2(\mathbb C_+).$
Next, we discuss the compactness of the operator $C_{\phi}$. 

Case 1. If $b_2> 0$, then $\phi(t_n) \notin \mathbb{R}$. Hence, $$\frac{ \vert E(\phi(t_n))\vert ^2 - \vert E(\overline{\phi(t_n)})\vert^2}{-4 i \Im(\phi(t_n))E(t_n) \Re(E'(t_n))}= \linebreak
\frac{\exp(2 d_2 b_1)(\exp(2d_1 b_2))- \exp(-2d_1 b_2))}{4 d_1 b_2}, $$
which is a positive constant.  

Case 2. If $b_2= 0$, then $\phi(t_n) \in \mathbb{R}$. Hence,$$\frac{-i \Im(E'(\phi(t_n)) \overline{E(\phi(t_n))})}{E(t_n) \Re(E'(t_n))}=\exp(2d_2b_1),$$
which is again a positive constant.

Thus, by Theorem \ref{Thm4.4}, the bounded operator $C_{\phi}$ is not compact on $H(E)$.
\end{example}

\begin{example}\label{example3.7}
Let $\phi(z)=az+b$, where $0<a \leq 1$ and $b =b_1+ib_2$ such that $b_2 \geq 0$. Let $E(z)= E(0) \exp(-idz)$ such that $d > 0$. It is easy to check that $E \in \mathcal{H}B$. Clearly, $\phi$ is an entire function such that $\phi(\mathbb{C}_+) \subseteq \mathbb{C}_+$. Also, $\sup_{z \in \mathbb{C}_+} \frac{\Im(z)}{\Im(\phi(z))}= \frac{1}{a}$. Observe that 
$$
\bigg\vert \frac{E \circ \phi}{E}(z) \bigg\vert = \bigg\vert\frac{E(az+b)}{E(z)} \bigg\vert= \exp(d \Im(b)+d(a-1)y) <  \exp(d \Im(b)).
$$
Hence $\frac{E \circ \phi}{E} \in H^{\infty}(\mathbb{C}_+)$. Thus, in the case $b_2=0$, we have $\phi^{\#}=\phi$ on $\mathbb{C}_+$; consequently, Theorem \ref{thm3.3} implies that the operator $C_{\phi}$ is bounded on $H(E)$. It remains to consider the case \(b_2>0\). In view of the proof of Theorem \ref{thm3.3}, it is enough to verify that $\frac{(f\circ \phi)^{\#}}{E}\in H^2(\mathbb{C}_+)$
for every $f\in H(E)$. This will prove the boundedness of $C_{\phi}$ on $H(E)$. For \(f\in H(E)\), we have \(\phi^\#(z)=az+\overline b\), and hence
\[
    \frac{(f\circ\phi)^\#(z)}{E(z)}
    =\frac{f^\#(az+\overline b)}{E(0)e^{-idz}}
    =C e^{idz}f^\#(az+\overline b),
\]
where \(C\neq 0\) is a constant. Since \(f^\#\) is of exponential type at most
\(d\) and belongs to \(L^2(\mathbb R)\), the Plancherel-Polya estimate (see \cite[Theorem 6.7.1]{Boas}, \cite[Lecture 7]{Levin}) gives
\[
\int_{\mathbb R}\left|
    \frac{(f\circ\phi)^\#(x+iy)}{E(x+iy)}
\right|^2dx
\leq
\frac{C}{a}e^{-2dy}e^{2d|ay-b_2|}
\int_{\mathbb R}|f^\#(t)|^2dt .
\]
Since \(0<a\leq 1\), the factor \(e^{-2dy}e^{2d|ay-b_2|}\) is uniformly
bounded for \(y>0\). Therefore $\frac{(f\circ\phi)^\#}{E}\in H^2(\mathbb C_+).$

Next, we discuss the compactness of the operator $C_{\phi}$. 

Case 1. If $b_2> 0$, then $\phi(t_n) \notin \mathbb{R}$. Hence, $$\frac{ \vert E(\phi(t_n))\vert ^2 - \vert E(\overline{\phi(t_n)})\vert^2}{-4 i \Im(\phi(t_n))E(t_n) \Re(E'(t_n))}= \linebreak
\frac{\exp(2d b_2))- \exp(-2d b_2)}{4 d b_2}, $$
which is a positive constant.  

Case 2. If $b_2= 0$, then $\phi(t_n) \in \mathbb{R}$. Hence,$$\frac{-i \Im(E'(\phi(t_n)) \overline{E(\phi(t_n))})}{E(t_n) \Re(E'(t_n))}=1,$$
which is again a positive constant.

Thus, by Theorem \ref{Thm4.4}, the bounded operator $C_{\phi}$ is not compact on $H(E)$.
\end{example}

The following example shows that the affine characterization of bounded composition operators on Paley-Wiener spaces does not extend to arbitrary de Branges spaces.

\begin{example}
Let $N_n=2^{2^n},~\varepsilon_n=e^{-N_n},~ \lambda_n=N_n-i\varepsilon_n,~ n\in\mathbb N$.
Fix \(\tau>0\), and define
\[
         E(z)=e^{-i\tau z}
        \prod_{n=1}^{\infty}
        \left(1-\frac{z}{\lambda_n}\right)e^{h_nz},
\] where $h_n=\frac{N_n}{N_n^2+\varepsilon_n^2}.$
The function \(E\) is an entire function in the Hermite-Biehler class. Indeed
\[
        \bigg\vert \frac{E^{\#}(z)}{E(z)}\bigg\vert
        =
        \bigg\vert e^{2i\tau z}
        \prod_{n=1}^{\infty}
        \frac{1-z/\overline{\lambda_n}}{1-z/\lambda_n}\bigg\vert <1, \quad \text{for all}~ z\in\mathbb C_+.
\]
For more details on the factorization of functions in the class $\mathcal{H}B$, see \cite[Section 2.2]{Baranov}.
Now consider the affine symbol $\phi(z)=z+1$. We claim that the corresponding composition operator $C_\phi $
is not bounded on \(H(E)\).

Suppose, on the contrary, that \(C_\phi\) is bounded on \(H(E)\). Let \(K_w\)
denote the reproducing kernel of \(H(E)\) at \(w\). Then $C_\phi^*K_w=K_{\phi(w)}=K_{w+1}$. Indeed, for every \(f\in H(E)\), $
        \langle f,C_\phi^*K_w\rangle_{H(E)}
        =
        \langle C_\phi f,K_w\rangle_{H(E)}
        =
        f(w+1)
        =
        \langle f,K_{w+1}\rangle_{H(E)}$.
Therefore, if \(C_\phi\) is bounded, then there exists \(M>0\) such that $
        \|K_{w+1}\|\leq M\|K_w\|$, $ w\in\mathbb C$.
In particular,
\begin{equation}\label{4.3}
        \sup_{n\in\mathbb N}
        \frac{\|K_{N_n+1}\|}{\|K_{N_n}\|}<\infty.
\end{equation}

We shall show that this is impossible. Let \(\varphi_E\) denote the phase
function associated with \(E\) (see \cite[Problem 48]{Brange}). For real \(x\), one has
\[
        \|K_x\|^2
        =
        K_x(x)
        =
        \frac{1}{\pi}\varphi_E'(x)|E(x)|^2.
\]
For the above choice of \(E\), the derivative of phase function associated with $E$ is given by
\[
        \varphi_E'(x)
        =
        \tau+
        \sum_{k=1}^{\infty}
        \frac{\varepsilon_k}{(x-N_k)^2+\varepsilon_k^2},
        \qquad x\in\mathbb R.
\]
Hence, at \(x=N_n\), the \(n\)-th summand gives $\varphi_E'(N_n)\geq \frac{1}{\varepsilon_n}.$
On the other hand,
\[
        \varphi_E'(N_n)
        =
        \tau+\frac{1}{\varepsilon_n}
        +
        \sum_{k\neq n}
        \frac{\varepsilon_k}{(N_n-N_k)^2+\varepsilon_k^2}.
\]
Since \(N_n=2^{2^n}\), the sequence \(\{N_n\}\) is strictly increasing and
sparse. In particular, for \(k\neq n\), $|N_n-N_k|\geq 1.$
Thus $(N_n-N_k)^2+\varepsilon_k^2\geq 1,$
and hence
\[
        \sum_{k\neq n}
        \frac{\varepsilon_k}{(N_n-N_k)^2+\varepsilon_k^2}
        \leq
        \sum_{k=1}^{\infty}\varepsilon_k.
\]
Since
\[
        \sum_{k=1}^{\infty}\varepsilon_k
        =
        \sum_{k=1}^{\infty}e^{-N_k}<\infty,
\]
there exists a constant \(A>0\) such that $\varphi_E'(N_n)
        \leq
        \tau+\frac{1}{\varepsilon_n}+A.$
Since \(0<\varepsilon_n\leq 1\), we get $\tau+A\leq \frac{\tau+A}{\varepsilon_n}.$
Therefore, for some constant \(C_1>0\), $\varphi_E'(N_n)\leq \frac{C_1}{\varepsilon_n}, n\in\mathbb N.$
On the other hand, $\varphi_E'(N_n+1)\geq \tau.$
Next, we estimate the quotient \(E(N_n+1)/E(N_n)\). Since \(N_n\) is real,
we have
\[
\begin{aligned}
        \left|
        \frac{E(N_n+1)}{E(N_n)}
        \right|
        &=
        \left|
        \frac{e^{-i\tau(N_n+1)}}{e^{-i\tau N_n}}
        \right|
        \prod_{k=1}^{\infty}
        \left|
        \frac{1-\frac{N_n+1}{\lambda_k}}
        {1-\frac{N_n}{\lambda_k}}
        \right|
        e^{h_k}  \\
        &=
        e^{\sum_{k=1}^{\infty}h_k}
        \prod_{k=1}^{\infty}
        \left|
        \frac{N_n+1-\lambda_k}{N_n-\lambda_k}
        \right|.
\end{aligned}
\]
The exponential factor \(e^{\sum h_k}\) is a positive constant independent of \(n\), because $\sum_{k=1}^{\infty}h_k<\infty$.
We now estimate the product. The \(k=n\) factor is
\[
        \left|
        \frac{N_n+1-\lambda_n}{N_n-\lambda_n}
        \right|
        =
        \left|
        \frac{1+i\varepsilon_n}{i\varepsilon_n}
        \right|
        =
        \frac{\sqrt{1+\varepsilon_n^2}}{\varepsilon_n}
        \geq
        \frac{1}{\varepsilon_n}.
\]

For \(k<n\), put \(d_{k,n}=N_n-N_k>0\). Then $N_n+1-\lambda_k = d_{k,n}+1+i\varepsilon_k,$
and $N_n-\lambda_k=d_{k,n}+i\varepsilon_k.$
Therefore
\[
        \left|
        \frac{N_n+1-\lambda_k}{N_n-\lambda_k}
        \right|
        =
        \frac{\sqrt{(d_{k,n}+1)^2+\varepsilon_k^2}}
        {\sqrt{d_{k,n}^2+\varepsilon_k^2}}
        \geq 1.
\]

For \(k>n\), put \(d_{n,k}=N_k-N_n>0\). Then $N_n+1-\lambda_k=-(d_{n,k}-1)+i\varepsilon_k,$
and $ N_n-\lambda_k= -d_{n,k}+i\varepsilon_k.$
Thus
\[
        \left|
        \frac{N_n+1-\lambda_k}{N_n-\lambda_k}
        \right|
        =
        \frac{\sqrt{(d_{n,k}-1)^2+\varepsilon_k^2}}
        {\sqrt{d_{n,k}^2+\varepsilon_k^2}}.
\]
Since \(d_{n,k}>1\) and \(0<\varepsilon_k<1\), we have $\sqrt{(d_{n,k}-1)^2+\varepsilon_k^2}
        \geq d_{n,k}-1$
and $\sqrt{d_{n,k}^2+\varepsilon_k^2}
        \leq d_{n,k}+1.$
Hence
\[
        \left|
        \frac{N_n+1-\lambda_k}{N_n-\lambda_k}
        \right|
        \geq
        \frac{d_{n,k}-1}{d_{n,k}+1}
        =
        1-\frac{2}{d_{n,k}+1}
        \geq
        1-\frac{2}{d_{n,k}}.
\]
Since \(N_k=2^{2^k}\), for \(k>n\) we have $N_k\geq N_{n+1}=N_n^2.$
Thus, for all \(n\) large enough, $d_{n,k}=N_k-N_n\geq \frac{N_k}{2}.$
Consequently, $\frac{2}{d_{n,k}}\leq \frac{4}{N_k}.$
Therefore
\[
        \left|
        \frac{N_n+1-\lambda_k}{N_n-\lambda_k}
        \right|
        \geq
        1-\frac{4}{N_k}.
\]
Since $\sum_{k=1}^{\infty}\frac{1}{N_k}<\infty,$
the infinite product $\prod_{k=2}^{\infty}\left(1-\frac{4}{N_k}\right)$
converges to a positive number. Hence there exists a constant
\(\delta>0\), independent of \(n\), such that
\[
        \prod_{k>n}
        \left|
        \frac{N_n+1-\lambda_k}{N_n-\lambda_k}
        \right|
        \geq \delta.
\]
Combining the estimates for \(k<n\), \(k=n\), and \(k>n\), we obtain a
constant \(c>0\), independent of \(n\), such that
\[
        \left|
        \frac{E(N_n+1)}{E(N_n)}
        \right|
        \geq
        \frac{c}{\varepsilon_n},
        \qquad n\in\mathbb N.
\]

Consequently,
\[
\begin{aligned}
        \frac{\|K_{N_n+1}\|^2}{\|K_{N_n}\|^2}
        &=
        \frac{\varphi_E'(N_n+1)}{\varphi_E'(N_n)}
        \left|
        \frac{E(N_n+1)}{E(N_n)}
        \right|^2   \\
        &\geq
        \frac{\tau}{C_1/\varepsilon_n}
        \frac{c^2}{\varepsilon_n^2}
        =
        \frac{C_2}{\varepsilon_n},
\end{aligned}
\]
where \(C_2>0\) is independent of \(n\). Since $
        \varepsilon_n=e^{-N_n}\to 0$,
we obtain
\[
        \frac{\|K_{N_n+1}\|}{\|K_{N_n}\|}
        \longrightarrow \infty.
\]
This contradicts \eqref{4.3}. Therefore \(C_\phi\) is not bounded on \(H(E)\).

\end{example}

\begin{rmk}
Chacon and Gimenez proved in \cite[Theorem~2.4]{Chacon} that, for a nonconstant entire function \(\phi\), the composition operator \(C_\phi\) is bounded on a Paley-Wiener space if and only if $\phi(z)=az+b,~ a\in\mathbb R,~0<|a|\leq 1$. The preceding example shows that this characterization does not extend to arbitrary de Branges spaces which are not Paley-Wiener spaces. Indeed, in the above example the symbol \(\phi(z)=z+1\) is affine with \(a=1\), but the corresponding composition operator \(C_\phi\) is not bounded on \(H(E)\).
\end{rmk}

In the next example, we illustrate that compact composition operators do exist in suitable de Branges spaces.

\begin{example}
Let \(a>0\) and $E(z)=e^{-iaz}(z+i)$. If \(z=x+iy\in \mathbb C_{+}\), then
\[
\frac{|E(\overline{z})|}{|E(z)|}
=
e^{-2ay}\frac{|z-i|}{|z+i|}
<1 .
\]
Thus $E \in \mathcal{H}B$. 
Next, we show that the constant function \(\mathbf{1}\) belongs to \( H(E)\). Since \(\mathbf{1}^{\#}=\mathbf{1}\), it is enough to check that $\frac{\mathbf{1}}{E(z)}=\frac{e^{iaz}}{z+i}$ belongs to \(H^{2}(\mathbb C_{+})\). Indeed,
\[
\sup_{y>0}\int_{\mathbb R}
\left|
\frac{e^{ia(x+iy)}}{x+i(y+1)}
\right|^{2}\,dx
=
\sup_{y>0}
e^{-2ay}
\int_{\mathbb R}
\frac{dx}{x^{2}+(y+1)^{2}}
<\infty .
\]
Thus \(\mathbf{1}/E\in H^{2}(\mathbb C_{+})\). Therefore, $\mathbf{1}\in H(E)$. Now define $\phi(z)=i$, for $z\in \mathbb C$ .
Then \(\phi\) is an entire function and the corresponding composition operator is given by
\[
C_{\phi}f=f\circ \phi=f(i)\cdot \mathbf{1}, \qquad f\in  H(E).
\]
Since point evaluations are bounded in every de Branges space and since
\(\mathbf{1}\in H(E)\), it follows that \(C_{\phi}\) is a bounded operator on
\( H(E)\). Moreover, the range of \(C_{\phi}\) is a one-dimensional subspace spanned by $\mathbf{1}$. Hence \(C_{\phi}\) is a rank-one operator and therefore compact.
Thus \(C_{\phi}\) is a nonzero compact composition operator on the de Branges space \( H(E)\). 

In general, let $H(E)$ be any infinite dimensional de Branges space and $\phi(z)$ be any constant function on $\mathbb{C}$. Then the corresponding composition operator $C_{\phi}$ on $H(E)$ is compact if and only if constant function $\mathbf{1} \in H(E)$.
\end{example}

\begin{rmk}It is known from the work of Chacon and Gimenez \cite[Corollary 2.5]{Chacon} that no bounded composition operator on a Paley-Wiener space is compact. The above example shows that this phenomenon does not extend to arbitrary de Branges spaces. Indeed, on a de Branges space $H(E)$ which is not a Paley-Wiener space, constant symbols may induce compact composition operators, provided that the constant function $\mathbf{1}$ belongs to $H(E)$.
\end{rmk}

\section{Composition operators on regular de Branges spaces}\label{sec4}

Recall that the de Branges space $H(E)$ is said to be regular if $H(E)$ is closed under the generalized backward shift operator $R_{\alpha}$ for every complex number $\alpha$. First, we present some elementary results that will be used to prove the boundedness of composition operators. 




\begin{lemma}\label{thm4.4}
The de Branges space $H(E)$ of entire functions is regular if and only if $\frac{E^{-1}}{\rho_i} \in H^2(\mathbb{C}_+)$.
\end{lemma}
\begin{proof}
See Lemma 3.18 in \cite{multi}.
\end{proof}

\begin{lemma}\label{lemma4.5} 
Let $H(E)$ be a regular de Branges space. Then the following holds:
\begin{itemize}
\item[(1)] $E$ is of exponential type.
\item[(2)] $\int_{-\infty}^{\infty} \frac{\log^+ \vert E(t)\vert}{1+t^2}dt < \infty$.
\end{itemize}
Moreover, every function $f \in H(E)$ is of exponential type less than or equal to the exponential type of $E$.
\end{lemma}
\begin{proof}
See Proposition 2 and Exercise 13 in \cite{DMcK}.
\end{proof}

Now, we discuss the boundedness of composition operators on regular de Branges spaces. The following two theorems are motivated by \cite{Chacon}, where the boundedness of composition operators is discussed on the Paley-Wiener spaces.

\begin{thm}\label{thm4.7}
Let $H(E)$ be a regular de Branges space and $\phi$ be a non constant entire function. If the operator $C_{\phi}$ is bounded on $H(E)$, then $\phi$ is affine.
\end{thm}
\begin{proof}
Since, $C_{\phi}$ is bounded on $H(E)$, the functions of the form $K_{\alpha}\circ \phi$ are in $H(E)$ where $K_{\alpha}$ are the reproducing kernels of $H(E)$. Here, $H(E)$ is regular, so by Lemma \ref{lemma4.5}, $K_{\alpha}$ and $K_{\alpha}\circ \phi$ are of exponential type and hence of order $1$. Now, by Theorem \ref{polya}, we get that $\phi$ is a polynomial. Now, following the same proof technique as in \cite[Lemma 2.3]{Chacon}, we see that $\phi$ is affine.
\end{proof}

\begin{thm}\label{thm4.8}
Let $H(E)$ be a regular de Branges space. If $\phi(z)= az+b$, where $0< a \leq 1$, $b=b_1+ib_2$, $b_1 \in \mathbb{R}, b_2 \geq 0$, and $\frac{E(az+b)}{E(z)} \in L^{\infty}(\mathbb{R})$, then the operator $C_{\phi}$ is bounded on $H(E)$. 
\end{thm}
\begin{proof}
Observe that $\phi$ is an entire function such that $\phi(\mathbb{C}_+) \subseteq \mathbb{C}_+$. Moreover,
$
\sup_{z \in \mathbb{C}_+} \frac{\Im(z)}{\Im(\phi(z))} = \frac{1}{a}.
$
Since the space $H(E)$ is regular, it follows from Lemma~\ref{thm4.4} that
$
\frac{E^{-1}}{\rho_i} \in H^{2}(\mathbb{C}_+) \subset N^{+}(\mathbb{C}_+),
$
and also $\rho_i \in H^{2}(\mathbb{C}_+) \subset N^{+}(\mathbb{C}_+)$.
Hence,
$
\frac{E^{-1}}{\rho_i}\,\rho_i \in N^{+}(\mathbb{C}_+),
$
which implies that $\mathrm{mt}(1/E) \leq 0$.
Since $H(E)$ is regular, Theorem \ref{thmA} together with Lemma \ref{lemma4.5} implies that
$E \in N(\mathbb{C}_+)$. In particular, $E \in N(\mathbb{C}_+)$ yields $E(az+b) \in N(\mathbb{C}_+)$.
Moreover, since $\lvert E^{\#} / E \rvert < 1$, we obtain
$
\frac{E^{\#}}{E} \in H^{\infty}(\mathbb{C}_+) \subset N(\mathbb{C}_+),
$
and therefore $\mathrm{mt}(E^{\#} / E) \leq 0$. Consequently,
$
\mathrm{mt}(E^{\#}) \leq - \mathrm{mt}(1/E) = \mathrm{mt}(E).
$
By Theorem \ref{thmB}, the mean type of $E$ is equal to the exponential type of $E$.
Since both $E(az+b)$ and $1/E(z)$ belong to $N(\mathbb{C}_+)$, it follows that
$
\frac{E(az+b)}{E(z)} \in N(\mathbb{C}_+).
$
Furthermore,
\begin{align*}
\mathrm{mt}\!\left(\frac{E(az+b)}{E(z)}\right)
&= \mathrm{mt}(E(az+b)) + \mathrm{mt}\!\left(\frac{1}{E(z)}\right) \\
&= |a|\,\mathrm{mt}(E) - \mathrm{mt}(E) \\
&= (|a|-1)\,\mathrm{mt}(E) \leq 0.
\end{align*}
By Lemma \ref{lemma4.5}, $E(z)$ has no zeros on the real line. Hence, $\frac{E(az+b)}{E(z)}$ has a continuous extension to the real line. Thus,
$
\frac{E(az+b)}{E(z)} \in N^{+}(\mathbb{C}_+).
$
Since it is assumed that
$
\frac{E(az+b)}{E(z)} \in L^{\infty}(\mathbb{R}),
$
and since $N^{+}(\mathbb{C}_+) \cap L^{\infty}(\mathbb{R}) = H^{\infty}(\mathbb{C}_+)$, we conclude that
$
\frac{E \circ \phi}{E} \in H^{\infty}(\mathbb{C}_+).
$
Thus, in the case $b_2=0$, the operator $C_{\phi}$ is bounded on $H(E)$ by Theorem \ref{thm3.3}. Next, we consider the case when $b_2>0$. By the proof of Theorem \ref{thm3.3}, it is sufficient to show that $\frac{(f \circ \phi)^{\#}}{E} \in H^2(\mathbb{C}_+)$ for every $f \in H(E)$. Since \(a\in\mathbb R\), we have
\(\phi^\#(z)=az+\overline b\), and hence $
    \frac{(f\circ\phi)^\#(z)}{E(z)}
    =
    \frac{f^\#(az+\overline b)}{E(z)}.$
For \(t\in\mathbb R\), we have $
    \left|\frac{f^\#(at+\overline b)}{E(t)}\right|
    =
    \left|\frac{f(at+b)}{E(t)}\right|.$
Therefore, by Remark \ref{rmk 5.5} given below, we get that
\[
    \int_{\mathbb R}
    \left|
        \frac{f^\#(at+\overline b)}{E(t)}
    \right|^2dt
    <\infty .
\]
Moreover, by regularity of \(H(E)\), Lemma \ref{lemma4.5} and the same argument used above, we obtain $
    \frac{f^\#(az+\overline b)}{E(z)}\in N^+(\mathbb C_+).$
Thus $
    \frac{(f\circ\phi)^\#}{E}
    \in N^+(\mathbb C_+)\cap L^2(\mathbb R)
    =
    H^2(\mathbb C_+).$
\end{proof}

\begin{rmk}\label{rmk 5.5}
For the space $H(E)$ and the function $\phi(z)$ as defined in the above theorem, note that for all $f \in H(E)$, 
\begin{align*}
\Vert C_{\phi}f \Vert^2 &= \int_{-\infty}^{\infty}\bigg\vert \frac{f \circ \phi}{E}(t)\bigg\vert^2 dt\\
&= \int_{-\infty}^{\infty}\bigg\vert \frac{f (at+b)}{E(t)}\bigg\vert^2 dt\\
&= \int_{-\infty}^{\infty}\bigg\vert \frac{f (at+b)}{E(at+b)}\bigg\vert^2 \bigg\vert \frac{E(at+b)}{E(t)}\bigg\vert^2 dt\\
&\leq \alpha^2 \int_{-\infty}^{\infty}\bigg\vert \frac{f (at+b_1+ib_2)}{E(at+b_1+ib_2)}\bigg\vert^2 dt\\
&\leq \frac{\alpha^2}{a} \int_{-\infty}^{\infty}\bigg\vert \frac{f (x+ib_2)}{E(x+ib_2)}\bigg\vert^2 dx\\
&\leq \frac{\alpha^2}{a}\exp(2 \sigma b_2) \int_{-\infty}^{\infty}\bigg\vert \frac{f (x)}{E(x)}\bigg\vert^2 dx\\
&= \frac{\alpha^2}{a}\exp(2 \sigma b_2) \Vert f \Vert^2,
\end{align*}
which implies that $$\Vert C_{\phi} \Vert \leq \frac{\alpha}{\sqrt{a}}\exp( \sigma b_2),$$ where $\alpha= \sup_{t \in \mathbb{R}}\big\vert \frac{E(at+b)}{E(t)}\big\vert$, $\sigma$ is the exponential type of $E$, and the last inequality follows by the Plancherel-Polya theorem (see \cite[Lecture 7]{Levin}).
\end{rmk}

\section{Concluding remarks}

In this paper, we have studied composition operators associated with de Branges
spaces of entire functions. The first part of the paper is devoted to composition
operators from model spaces in the upper half-plane into the Hardy Hilbert space
\(H^2(\mathbb C_+)\). For an inner function \(\chi\), we obtained a necessary
condition for the boundedness of
\[
        C_\phi:H(\chi)\to H^2(\mathbb C_+)
\]
in terms of the reproducing kernels of \(H(\chi)\) and \(H^2(\mathbb C_+)\).
We also obtained a sufficient condition in terms of the finite angular derivative
of the symbol \(\phi\) at infinity. Under an additional natural hypothesis on
\(\chi\circ\phi\), these conditions give a boundedness characterization. We also
proved a non-compactness result for such bounded composition operators and
illustrated the result by examples involving meromorphic inner functions in the
upper half-plane.

Using the unitary correspondence between a de Branges space \(H(E)\) and the
model space \(H(\chi)\), where \(\chi=E^\#/E\), we then transferred the
model-space results to the setting of de Branges spaces. This allowed us to
obtain sufficient conditions for the boundedness of composition operators on
\(H(E)\) in terms of the multiplier-type condition
\[
        \frac{E\circ\phi}{E}\in H^\infty(\mathbb C_+)
\]
together with the finite angular derivative condition for \(\phi\). We also
derived a necessary condition for boundedness involving the Hermite-\-Biehler
function \(E\), the associated inner function \(E^\#/E\), and the reproducing
kernels of the corresponding de Branges space. Thus, the boundedness problem
for \(C_\phi\) on \(H(E)\) is related not only to the growth of the symbol
\(\phi\), but also to the interaction between \(\phi\) and the structural
function \(E\).

We also studied compactness of composition operators on de Branges spaces.
Using the orthonormal basis of reproducing kernels associated with the real
zeros of the function
\[
        A(z)=\frac{E(z)+E^\#(z)}{2},
\]
we obtained a sufficient condition which guarantees that a bounded composition
operator on \(H(E)\) is not compact. Several examples were provided to show how
these criteria apply. In particular, we showed that bounded affine symbols may
give rise to non-compact composition operators on suitable de Branges spaces.
We also exhibited examples showing that the general de Branges setting differs
substantially from the Paley-Wiener setting: an affine symbol need not always
induce a bounded composition operator on an arbitrary de Branges space, while
compact composition operators may exist in suitable non-Paley-Wiener
de Branges spaces.

Finally, for regular de Branges spaces, we obtained an affine rigidity theorem.
More precisely, if \(H(E)\) is regular and \(C_\phi\) is bounded for a
non-constant entire symbol \(\phi\), then \(\phi\) must be affine. This extends
the affine-symbol phenomenon known for Paley-Wiener spaces to the class of
regular de Branges spaces. At the same time, the examples in the paper show that
the larger class of de Branges spaces is more flexible than the Paley-Wiener
class. Therefore, the results obtained here should be viewed as a first step
towards understanding how the geometry and zero distribution of the
Hermite-Biehler function \(E\) influence composition operators on \(H(E)\).

The motivation for this study comes from the fact that de Branges spaces form
a natural and important class of Hilbert spaces of entire functions, with
connections to spectral theory, differential equations and prediction theory.
Since Paley-Wiener spaces are particular examples of de Branges spaces, it is
natural to ask which properties of composition operators on Paley-Wiener spaces
remain true in the larger de Branges setting. Our results show that the affine
rigidity phenomenon persists for regular de Branges spaces. However, the general
de Branges setting is not merely a repetition of the Paley-Wiener case. In
particular, an affine symbol which behaves well in the Paley-Wiener setting
need not induce a bounded composition operator on an arbitrary de Branges space.
Moreover, while bounded composition operators on Paley-Wiener spaces are not
compact, compact composition operators may exist on suitable non-Paley-Wiener
de Branges spaces.

The present work does not give a complete characterization of boundedness and
compactness of \(C_\phi\) on an arbitrary de Branges space \(H(E)\). In particular, the gap between the necessary and sufficient conditions
obtained in the model-space formulation is not completely closed in full
generality. The results
obtained here indicate that the problem is governed not only by the growth of
the symbol \(\phi\), but also by the finer structure of the Hermite-Biehler
function \(E\), the distribution of its zeros, the associated phase function,
and the embedding properties of the corresponding model space. Thus, although
the affine form of the symbol is forced in the regular case, the general
boundedness and compactness problems for de Branges spaces remain open.

We conclude by recording some natural problems and directions for further
investigation. Several of these directions were kindly suggested to us by one of the
referees, and we are grateful for these valuable suggestions.

\begin{itemize}
    \item It would be interesting to characterize the boundedness of
    \(C_\phi\) on \(H(E)\) directly in terms of the zeros of the
    Hermite-Biehler function \(E\). Equivalently, one may ask whether the
    boundedness of \(C_\phi\) can be described in terms of the phase function
    associated with \(E\) or in terms of the growth of the reproducing kernels
    of \(H(E)\).

    \item A second natural problem is to characterize compactness of
    composition operators on de Branges spaces. The examples in this paper
    show that compactness phenomena in general de Branges spaces are different
    from those in Paley-Wiener spaces. It would therefore be useful to
    determine which structural properties of \(E\) allow compact composition
    operators to exist.

    \item Another direction is to study boundedness and compactness through
    Carleson measure methods. If \(\phi:\mathbb C_+\to\mathbb C_+\) is
    analytic and \(\mu_\phi\) is the pull-back measure induced by the
    boundary values of \(\phi\), then the boundedness of
    \[
        C_\phi:H(\chi)\to H^2(\mathbb C_+)
    \]
    is equivalent to the boundedness of the natural embedding
    \[
        \mathcal J_{\mu_\phi}:H(\chi)\to L^2(\mu_\phi).
    \]
    Equivalently, \(\mu_\phi\) is a Carleson measure for the model space
    \(H(\chi)\). In future work, we plan to develop this point of view further
    in order to obtain boundedness and compactness characterizations for
    composition operators on de Branges spaces in terms of Carleson and
    vanishing Carleson measures.

    \item Since model spaces and Hardy spaces in the upper half-plane are
    isometrically equivalent to their counterparts on the unit disc via the
    Cayley transform, it would be worthwhile to compare the present results
    with the corresponding disc formulation.

    \item It would also be interesting to investigate whether
    Aleksandrov-Clark measures associated with the inner function
    \(E^\#/E\) provide a more intrinsic description of the boundedness of
    composition operators on de Branges spaces.

    \item Finally, one may study related operators, such as weighted
    composition operators and Hausdorff operators, in the setting of
    de Branges spaces.
\end{itemize}

In particular, the following problem remains open.

\begin{prob}
Let \(E\in HB\) and let \(\phi\) be an entire function. Characterize, in terms
of the Hermite-Biehler function \(E\), the zero distribution of \(E\), the
phase function of \(E\), and the pull-back measure associated with \(\phi\),
when the composition operator \(C_\phi\) is bounded, respectively compact, on
\(H(E)\).
\end{prob}

It would also be natural to determine whether suitable non-regular de Branges
spaces admit bounded composition operators induced by non-constant non-affine
entire symbols, or whether additional structural assumptions on \(E\) force
every bounded composition operator to have an affine symbol.

\noindent

\vspace{.2in}

\noindent \textbf{Acknowledgements:}
The authors gratefully acknowledge the anonymous referees for providing several valuable improvements and suggestions on the earlier version of the manuscript, which have significantly contributed to the development of the revised version.
\vspace{.3cm}

\noindent \textbf{Funding:}
The research of the first author is supported by the FIST program of the Department of Science and Technology, Government of India, Reference No. SR/FST/MS-I/2018/22(C). The research of the second author is supported by the Post-doctoral fellowship funded under DAE plan project RIN 4001 (NISER Bhubaneswar). The research of the third author is supported 
by the MATRICS grant of SERB (MTR/2023/001324).

\vspace{.3cm}

\noindent \textbf{Author Contributions:} All authors contributed equally towards the paper.

\vspace{.3cm}

\noindent\textbf{Data availability:}\\
No data was used for the research described in the article.\\


 \noindent\textbf{Declarations:}
\vspace{.3cm}

\noindent\textbf{Conflict of interest:}\\
The authors declare that they have no conflict of interest.

\end{document}